\documentclass[a4paper,12pt]{article}
\usepackage{graphics}
\usepackage{epsf}

\usepackage{amsfonts}
\usepackage{amssymb}
\usepackage{epsfig}
\usepackage[latin1]{inputenc}
\usepackage{colortbl}

\usepackage{amsthm,amsmath,amssymb}

\newtheorem{theorem}{Theorem}

\newtheorem{remark}[theorem]{Remark}

\begin{document}
\title{Splitting methods for nonlinear Schr\"odinger equation without order reduction}
\author{{\sc C. Arranz-Sim\'on \thanks{Email: carlos.arranz@uva.es} {\sc and} B. Cano \thanks{Corresponding author. Email: bcano@uva.es}}  \\ \small
IMUVA, Departamento de Matem\'atica Aplicada,\\ \small Facultad de Ciencias, Universidad de
Valladolid,\\ \small Paseo de Bel\'en 7, 47011 Valladolid,\\ \small Spain
}
\date{}

%
%

\maketitle

\begin{abstract}
A technique is provided in this paper to integrate nonlinear Schr\"odinger equation with time-dependent Dirichlet boundary conditions with high-order Yoshida splittings which are based on Strang method. For that, a modification of Strang method is required in which the linear and stiff part of the equation is integrated with a rational-like version of midpoint rule for which the required boundary values can be calculated without resorting to any differentiation of data. Although Yoshida splitting (with real coefficients) cannot be applied to parabolic problems to obtain order higher than two because of stability, the modified Strang method is also applicable to such type of problems and local order $3$ and global order $2$ are also obtained without differentiation of data.
\end{abstract}



\section{Introduction}

It is well-known the phenomenom of order reduction which splitting methods show, even when integrating linear problems with homogeneous boundary conditions and solving each part in an exact way (see \cite{FOS} for the analysis of Lie-Trotter and Strang method.) Because of that, several techniques have been developed in the literature to avoid it. Up to our knowledge, the best result in this direction is the technique described in \cite{APo} for some semilinear wave problems with time-dependent boundary values, for which arbitrary high order can be observed. On the other hand, for multidimensional problems with commuting operators there are also techniques for splitting integrators which manage to improve accuracy reasonably \cite{ACJ,ACR0, GH}. However, there are other problems which are not of that type, like reaction-diffusion problems or nonlinear Schr\"odinger equation, for which local order $3$ and global order $2$ is the maximum which has been achieved in the literature when the splitting separates the linear and nonlinear part \cite{ACR, ACR1, CR, EO1, EO2}. What is more, in order to achieve local order $3$, numerical differentiation is required, which is well-known to be unstable for very small stepsizes. However, local and global order $2$ can be obtained with a summation-by-parts argument without any numerical differentiation of data.

On the other hand, a technique has been recently described in the literature to avoid order reduction of initial boundary value problems with rational-like methods \cite{ACP}. The main advantage of this strategy is that it does not require either numerical nor analytical differentiation of data, but of course boundary values are required and the higher the accuracy that wants to be achieved, the higher the number
of necessary boundary evaluations. The key of order reduction with splitting methods is the determination of the boundary values of the intermediate problems and the main idea of this paper is that the rational-like method which is based on the midpoint rule just requires the boundary values at the beginning and end of each step in order that no order reduction is observed. Those boundary values are easily calculable in terms of the data of the original problem and, in such a way, we get local order $3$ with the so-called modified Strang method without using any differentiation of data.

We focus on NLS equation in this paper  because, in such a case, the nonlinear part can be exactly solved (which simplifies the description of the technique and the analysis) but, more importantly, because splitting methods do not show a stability order barrier with this equation. (We notice that in parabolic problems that order barrier is two \cite{BC, Sh} when considering splitting methods with real coefficients.) In such a way, we manage to obtain arbitrary high order with Yoshida splittings \cite{Y} without any differentation of data.

The paper is structured as follows. Section 2 gives some preliminaries on Strang method applied to NLS equation. Section 3 describes the rational-like midpoint rule to solve the stiff part of Strang method without order reduction and conserving symmetry. In Section 4, the whole modified Strang method is analysed. Section 5 justifies the arbitrary order which is thus achieved with Yoshida splitting using the modified Strang method as a basis. Section 6 shows some numerical experiments which corroborate the results and, finally, in an appendix, the analysis for the modified Strang method in an abstract framework which includes reaction-diffusion problems is given.

\section{Preliminaries}

We will use the following notation for the solution of nonlinear Schr\"odinger equation
\begin{eqnarray}
\dot{u}(t)&=& i \big(A u(t)+f(|u(t)|^2) u\big), \quad 0 \le t \le T,\nonumber \\
u(0)&=&u_0,  \nonumber \\
\partial u(t)&=& g(t), \label{nlse}
\end{eqnarray}
where $A: H^2(\Omega)\to L^2(\Omega)$ denotes the operator which applies the second derivative in space in one dimension and the Laplacian in general for a bounded domain $\Omega$ with regular enough boundary. On the other hand, $f$ is a smooth enough real function, $u_0\in H^2(\Omega)$ and $\partial: H^2(\Omega) \to H^ {\frac{3}{2}}(\partial \Omega)$ is the Dirichlet operator.

In the following, we will use the operators $A_0=A|_{\mbox{ker}(\partial)}$ and $K: H^{\frac{3}{2}}(\partial \Omega) \to H^2(\Omega)$, so that $K g$ is the solution of the elliptic problem
\begin{eqnarray}
A u &=&0, \nonumber \\
\partial u&=&g, \nonumber
\end{eqnarray}
which is well-known to satisfy \cite{ACP}
\begin{eqnarray}
\|Kg \|_{L^2(\Omega)}\le C \|g\|_{H^{\frac{3}{2}}(\partial \Omega)},
\nonumber
\end{eqnarray}
for some constant $C$.
Moreover, if $g \in H^{\frac{3}{2}}(\partial \Omega)\cap L^{\infty}(\Omega)$, for some other constant it happens that \cite{S}
\begin{eqnarray}
\|Kg \|_{L^\infty(\Omega)}\le C \|g\|_{L^\infty(\partial \Omega)}.
\label{acotK2}
\end{eqnarray}
It is well-known that advancing a stepsize $\tau$ with Strang method to solve a general differential system of the form
$$\dot{u}=f_1(u)+f_2(u),$$
consists of the following
$$u_{n+1}=\phi_{\frac{\tau}{2}}^{f_2}\circ\phi_{\tau}^{f_1}\circ\phi_{\frac{\tau}{2}}^{f_2}(u^n),$$
where $\phi_\tau^{f_1}$ and $\phi_\tau^{f_2}$ are the exact flows after time $\tau$ of the respective systems
$$
\dot{u}=f_1(u), \quad \dot{u}=f_2(u).$$
In the case of
(\ref{nlse}), one of the terms in the splitting would be
\begin{eqnarray}
\dot{u}&=& i f(|u|^2) u, \label{nolinear}
\end{eqnarray}
which can be exactly solved because $|u|^2$ remains constant under this evolution \cite{WH}
$$
\frac{d}{dt} |u|^2= \dot{u} \bar{u}+ u \dot{\bar{u}}=i f(|u|^2) u  \bar{u}-i u f(|u|^2) \bar{u}=0,$$
so this can be solved as a linear system. More precisely, the solution of (\ref{nolinear}) from an initial condition $u_0$ would be
$$
u(t)=e^{i t f(|u_0|^2)}u_0.$$

On the other hand, for the solution of the other term of the splitting
\begin{eqnarray}
\dot{u}&=& i A u, \label{linear}
\end{eqnarray}
to be completely determined, apart from an initial condition, a boundary condition is also necessary. In that sense, in \cite{CR}, some boundaries are suggested which are of course related to the boundary $g$ in (\ref{nlse}) and for which no order reduction is observed. More precisely, second global order is observed when applying Strang method with those boundaries and solving exactly its various stages. If an analytic expression for $g$ is known, those boundaries can be calculated without resorting to numerical differentiation, although just local order $2$ is achieved and a summation-by-parts argument must be applied to justify the same global order. In any case, $\dot{g}$ is necessary and, if no analytic expression of $g$ is available, numerical differentiation would also be required. Finally, notice that, as justified in Remark 1 of \cite{CR}, if local order $3$ is searched for, numerical differentiation in space from the numerical solution and not directly from the data in (\ref{nlse}) would also be required.

\section{Midpoint rule to solve the stiff part of Strang method without order reduction and conserving symmetry}

In this section, without knowing the exact boundary in (\ref{linear}) which would lead to local order $3$ or more, we suggest a way to integrate it (again with local order $3$) just by knowing the exact values of the boundary at the beginning and end of integration of such a stage.
We notice that, at each step, such part of Strang method corresponds to solve
\begin{eqnarray}
\dot{w}^n(s)&=&i A w^n(s), \nonumber \\
\partial w^n(s)&=& g^n(s), \nonumber \\
w^n(0)&=& e^{i \frac{\tau}{2}f(|u_n|^2)} u_n, \label{midstrang}
\end{eqnarray}
where $g^n$ should be a function which satisfies
\begin{eqnarray}
g^n(0)&=& e^{i \frac{\tau}{2}f(|g(t_n)|^2)} g(t_n), \nonumber \\
g^n(\tau)&=& e^{-i \frac{\tau}{2}f(|g(t_{n+1})|^2)} g(t_{n+1}), \label{known}
\end{eqnarray}
so that the boundary fits perfectly with that to which we arrive advancing $\tau/2$ from $u_n$ with the nonlinear part of Schr\"odinger equation and to that to which we want to arrive when advancing $\tau/2$ from $w^n(\tau)$ with the same nonlinear part.

According to \cite{ACP}, in order to integrate (\ref{midstrang}) without order reduction with a rational-like method which is based on the implicit midpoint rule, for which the stability function is
\begin{eqnarray}
r(z)=-1+\frac{2}{1-\frac{z}{2}},
\label{r}
\end{eqnarray}
the following must be implemented
\begin{eqnarray}
\tilde{w}_{n+1}= r(i \tau A_0) \bar{w}_n- F_n(\tau)g^n(0+\cdot),  \nonumber
\end{eqnarray}
where $\tilde{w}_n=w^n(0)-K g^n(0)$ and
$$F_n(\tau) g^n =(I-i \frac{\tau}{2}A_0)^{-1} K \boldsymbol{\eta}^T  g^n(0+\tau \mathbf{d}),$$
so that
$\boldsymbol{\eta}^T g^n(0+\tau \mathbf{d})$ approximates $L\tau B (I- \tau B/2)^{-1} g^n(0+\cdot)$ with order $3$, where $B$ is the operator consisting of the first time derivative applied over functions in $C^1_{ub}([0,\infty), H^{\frac{3}{2}}(\partial \Omega))$ and $L: C_{ub}([0,\infty),H^{\frac{3}{2}}(\partial \Omega))\to H^{\frac{3}{2}}(\partial \Omega))$ denotes the delta operator
$$L h= h(0), \quad h \in C_{ub}([0,\infty),H^{\frac{3}{2}}(\partial \Omega)).$$
Then, the approximation to $w^n(\tau)$ would be given by $\tilde{w}_{n+1}+ K g^n(\tau)$.

According to Lemma 4.3 in \cite{AP}, $\boldsymbol{\eta}=(\eta_1,\eta_2,\eta_3)^T$ and $\mathbf{d}=(d_1,d_2,d_3)^T$ must be chosen so that
$$
d_1^j\eta_1+d_2^j \eta_2+d_3^j \eta_3=j! F_j, \quad 0 \le j \le 2,
$$
where $F_j$ are the first Taylor coefficients of $F(z)=z/(1-z/2)$ around $z=0$. We notice that, by taking $d_1=0, d_2=1/2, d_3=1$, we do not only manage to consider the known values (\ref{known}) but also that the coefficient $\eta_2=0$ so that the intermediate value $g^n(\tau/2)$ is in fact not necessary. Notice that the system to solve is
\begin{eqnarray}
\begin{array}{ccccccc}
\eta_1&+&\eta_2&+&\eta_3&=&0,  \\
&&  \eta_2/2&+&\eta_3&=&1,  \\
&& \eta_2/4 &+&\eta_3&=&1,
\end{array}
\end{eqnarray}
for which $\eta_1=-1$, $\eta_2=0$, $\eta_3=1$.

Plugging this together,
\begin{eqnarray}
\tilde{w}_{n+1}&=&[-I+2(I-\frac{i \tau}{2}A_0)^{-1}]\tilde{w}_n- (I-\frac{i \tau}{2}A_0)^{-1}K [g^n(\tau)-g^n(0)],
\nonumber
\end{eqnarray}
from what $w_{n+1}=\tilde{w}_{n+1}+K g^n(\tau) \approx w^n(\tau)$ would be
\begin{eqnarray}
w_{n+1}&=&-w^n(0)+K [g^n(0)+g^n(\tau)]
\nonumber \\
&&+(I-\frac{i \tau}{2} A_0)^{-1}\bigg[2 w^n(0)- K[g^n(0)+g^n(\tau)]\bigg].
\label{wn1}
\end{eqnarray}
According to \cite{ACP}, the error commited in this way satisfies the following when $w^n \in C^3([0,\tau],L^2(\Omega))$ and $g^n \in C^4([0,\tau],L^2(\partial \Omega))$
\begin{eqnarray}
\hspace{0.5cm}\|w_{n+1}-w^n(\tau)\|_{L^2(\Omega)}\le C \tau^3 (\max_{0\le j \le 3} \|{w^n}^{(j)}(0)\|_{L^2(\Omega)}+\max_{0\le j \le 4} \|{g^n}^{(j)}(0)\|_{L^2(\partial \Omega)}), \label{erlocmp}
\end{eqnarray}
for some constant $C$.

\subsection{Symmetry}
In this subsection, we will prove that advancing $-\tau$ from $w_{n+1}$ using the same formula in (\ref{wn1}), we would arrive at $w^n(0)$.
For that, it suffices to take into account that (\ref{wn1}) can be written as
\begin{eqnarray}
w_{n+1}=[-I +2(I-\frac{i \tau}{2}A_0)^{-1}]w^n(0)+[I-(I-\frac{i \tau}{2}A_0)^{-1}]K [g^n(0)+g^n(\tau)],
\label{sim1}
\end{eqnarray}
and that changing $\tau$ by $-\tau$, $w_{n+1}$ by $w^n(0)$ and $g^n(0)$ by $g^n(\tau)$, it follows that
\begin{eqnarray}
w^n(0)=[-I +2(I+\frac{i \tau}{2}A_0)^{-1}]w_{n+1}+[I-(I+\frac{i \tau}{2}A_0)^{-1}]K [g^n(0)+g^n(\tau)].
\label{sim2}
\end{eqnarray}
Then, by using the algebraic expressions
$$
(-1+\frac{2}{1+z})^{-1}=-1+\frac{2}{1-z}, \quad (-1+\frac{2}{1+z})^{-1}(1-\frac{1}{1+z})=-1+\frac{1}{1-z},$$
(\ref{sim1}) can be seen equivalent to (\ref{sim2}), and the symmetry follows.


\section{Modified Strang method}

In the following, we will denote the modified Strang method to that which consists of the calculation of $w_{n+1}$ through (\ref{wn1}) starting from $w^n(0)$ in (\ref{midstrang}) and then completing the step through the formula
\begin{eqnarray}
u_{n+1}=e^{i \frac{\tau}{2}f(|w_{n+1}|^2)} w_{n+1}. \label{finalu}
\end{eqnarray}
\subsection{Local error}
Third local order can be achieved in such a way, as the following theorem states for $\rho_n=\hat{u}_{n+1}-u(t_{n+1})$, where $\hat{u}_{n+1}$ is that obtained through (\ref{midstrang}),(\ref{wn1}) and (\ref{finalu}) when starting from $u_n=u(t_n)$.
\begin{theorem}
Let us assume that the solution $u$ of (\ref{nlse}) belongs to
\begin{eqnarray}
\quad C^3([0,T], L^2(\Omega))\cap C([0,T], H^6(\Omega))  \cap C^1([0,T], H^4(\Omega)) \cap C^2([0,T],L^{\infty}(\Omega)),
\label{reg2}
\end{eqnarray}
where $f\in C^2(\mathbb{R}, \mathbb{R})$ and $g\in C^2([0,T], H^{\frac{3}{2}}(\partial \Omega))$. Then, for small enough $\tau$,
\begin{eqnarray}
\|\rho_n\|_{L^2(\Omega)} &\le& C \tau^3 (\|u(t_n)\|_{H^6(\Omega)}+\|\dot{u}(t_n)\|_{H^4(\Omega)}
+\max_{0\le j \le 3} \|u^{(j)}(t_n)\|_{L^2(\Omega)}
\nonumber \\
&&\hspace{6cm}+\max_{0\le j \le 2}\|g^{(j)}(t_n)\|_{H^{\frac{3}{2}}(\partial \Omega)}),
\label{cotaloc}
\end{eqnarray}
for some constant $C$, which does not depend on $\tau$ either on $t_n$.
\label{thloc}
\end{theorem}

\begin{proof}
For the sake of simplicity, let us denote
\begin{eqnarray}
H_n=e^{i \frac{\tau}{2}f(|u(t_n)|^2)}u(t_n), \quad H_{n+1}=e^{-i \frac{\tau}{2}f(|u(t_{n+1})|^2)}u(t_{n+1}).
\label{HH}
\end{eqnarray}
Then, we consider $w^{n,b}(s)$ the polynomial of second degree in $s$ which takes the value $H_n$ at $s=0$, $H_{n+1}$ at $s=\tau$ and a value $X_n$ to be determined at $s=\tau/2$. More precisely,
$$
w^{n,b}(s)=H_n+\frac{H_{n+1}-H_n}{\tau}s+\frac{2(H_{n+1}+H_n-2 X_n)}{\tau^2}s(s-\tau).$$
Let us then denote by $\hat{w}^n$ the solution of (\ref{midstrang}) with initial condition $\hat{w}^n(0)=H_n$ and $g^n(s)=\partial w^{n,b}(s)$, which obviously satisfies (\ref{known}). We notice that $\hat{w}^n(s)-w^{n,b}(s)$ satisfies the following
\begin{eqnarray}
\quad \stackrel{\cdot}{\overbrace{(\hat{w}^n-w^{n,b})}}(s)&=&i A (\hat{w}^n(s)-w^{n,b}(s))+i A H_n-\frac{4 X_n-3 H_n-H_{n+1}}{\tau} \nonumber \\
\quad &&+s \big[\frac{i}{\tau}A(-H_{n+1}-3 H_n+4 X_n)-\frac{4}{\tau^2}(H_{n+1}+H_n-2 X_n)\big] \nonumber \\
\quad &&+\frac{2 i s^2}{\tau^2}A(H_{n+1}+H_n-2 X_n), \nonumber \\
\quad (\hat{w}^n-w^{n,b})(0)&=&0, \nonumber \\
\quad \partial (\hat{w}^n-w^{n,b})(s)&=&0. \label{diferencia}
\end{eqnarray}
Then, if we take $X_n$ such that the second term of the right-hand side of the first line in (\ref{diferencia}) vanishes, i.e.
$$X_n=\frac{1}{4}[ i \tau A H_n+H_{n+1}+3 H_n],$$
it can be deduced that
\begin{eqnarray}
H_{n+1}+H_n-2 X_n&=&\frac{H_{n+1}-H_n}{2}-\frac{i \tau}{2} A H_n. \nonumber
\end{eqnarray}
We then notice that, by Taylor expansions in powers of $\tau$ and using (\ref{nlse}),
\begin{eqnarray}
H_n&=&u+\frac{i \tau}{2}f(|u|^2)u-\frac{\tau^2}{8}f(|u|^2)^2 u- \frac{i \tau^3}{48} f(|u|^2)^3 u +O(\tau^4), \nonumber \\
H_{n+1}-H_n&=& i \tau A u+\frac{\tau^2}{2}[\ddot{u}-i f(|u|^2) \dot{u}-i f'(|u|^2)(u \dot{\bar{u}}+\bar{u}\dot{u})u]+\tau^3 D, \nonumber
\end{eqnarray}
where, for the sake of brevity, we have suppressed the evaluation at $t_n$ and where
\begin{eqnarray}
D&=&\frac{1}{6}\stackrel{\dots}{u}-\frac{i}{4}f'(|u|^2)(\bar{u}\ddot{u}+2 \dot{u}\dot{\bar{u}}+u \ddot{\bar{u}})u-\frac{1}{4}f(|u|^2)f'(|u|^2)(u\dot{\bar{u}}+\bar{u}\dot{u})u \nonumber \\
&&-\frac{i}{2}f'(|u|^2)(u \dot{\bar{u}}+\bar{u} \dot{u})\dot{u}-\frac{1}{8}f(|u|^2)^2 \dot{u}-\frac{i}{4}f(|u|^2) \ddot{u} \nonumber \\
&&+\frac{i}{24}f(|u|^2)^3 u -\frac{i}{4} f''(|u|^2)(u \dot{\bar{u}}+\bar{u}\dot{u})^2 u +O(\tau). \nonumber
\end{eqnarray}
Using this, (\ref{diferencia}) can be written as
\begin{eqnarray}
\stackrel{\cdot}{\overbrace{(\hat{w}^n-w^{n,b})}}(s)&=&i A (\hat{w}^n(s)-w^{n,b}(s)) \nonumber \\
&&+s \big[-\frac{3 i \tau}{4}A^2 \big( f(|u(t_n)|^2) u(t_n)\big)-4 \tau D(t_n) +O(\tau^2)\big] \nonumber \\
&&+ s^2 \big[ -\frac{i}{2} A^3 u (t_n)+O(\tau) \big], \nonumber \\
(\hat{w}^n-w^{n,b})(0)&=&0, \nonumber \\
\partial (\hat{w}^n-w^{n,b})(s)&=&0. \label{diferenciab}
\end{eqnarray}
By using then the variation-of-constants formula and the definition of the $\varphi_j$-functions \cite{HO} , which are well-known to imply that
\begin{eqnarray}
\varphi_j( t A_0)=\frac{1}{t^j} \int_0^t
e^{(t-\tau)A_0}\frac{\tau^{j-1}}{(j-1)!}d\tau, \quad j \ge 1,
\nonumber
\end{eqnarray}
it follows that
\begin{eqnarray}
\hat{w}^n(s)-w^{n,b}(s)&=&s^2 \tau \varphi_2(i s A_0)\big[-\frac{3 i}{4}A^2 \big( f(|u(t_n)|^2) u(t_n)\big)-4 D(t_n) +O(\tau)\big]
\nonumber \\
 &&+s^3  \varphi_3(i s A_0)[-\frac{i}{2} A^3 u(t_n)+O(\tau)].
\label{exprphi}
\end{eqnarray}
Therefore, as $\|\varphi_2(i s A_0)\|_{L^2(\Omega)}$ and $\|\varphi_3(i s A_0)\|_{L^2(\Omega)}$ are bounded,
\begin{eqnarray}
\|\hat{w}^n(\tau)-H_{n+1}\|_{L^2(\Omega)}&=&\|\hat{w}^n(\tau)-w^{n,b}(\tau)\|_{L^2(\Omega)} \nonumber \\
&\le& C \tau^3 (\|u(t_n)\|_{H^6(\Omega)}+\|\dot{u}(t_n)\|_{H^4(\Omega)}+\max_{0\le j \le 3} \|u^{(j)}(t_n)\|_{L^2(\Omega)}),
\nonumber
\end{eqnarray}
where we have used that, from (\ref{nlse}),
$$A^2 \big( f(|u|^2)u\big)=-i A^2 \dot{u}-A^3 u,$$
and the assumed regularity (\ref{reg2}).

Considering now $\hat{w}_{n+1}$ as the approximation to $\hat{w}^n(\tau)$ given by the rational-like midpoint rule without order reduction, and using (\ref{erlocmp}) and the fact that
\begin{eqnarray}
&&\hat{w}^n(0)=H_n, \, \hat{w}^{n(1)}(0)=i A H_n, \, \hat{w}^{n(2)}(0)=-A^2 H_n, \, \hat{w}^{n(3)}(0)=-i A^3 H_n,\nonumber \\
&&g^n(0)=\partial H_n, \,  g^{n(1)}(0)=\partial [i A H_n],
\, g^{n(2)}(0)=-\partial A^2 u(t_n)+O(\tau), \, g^{n(j)}(0)=0, \, j=3,4, \nonumber
\end{eqnarray}
it happens that
\begin{eqnarray}
\|\hat{w}_{n+1}-H_{n+1}\|_{L^2(\Omega)}&\le& \|\hat{w}_{n+1}-\hat{w}^n(\tau)\|_{L^2(\Omega)}+\|\hat{w}^n(\tau)-H_{n+1}\|_{L^2(\Omega)}
\nonumber \\
&\le & C \tau^3 (\|u(t_n)\|_{H^6(\Omega)}+\|\dot{u}(t_n)\|_{H^4(\Omega)}+\max_{0\le j \le 3} \|u^{(j)}(t_n)\|_{L^2(\Omega)}
\nonumber \\
&& \hspace{2cm}+ \max_{0\le j \le 2}\| \partial A^j u(t_n)\|_{H^{\frac{3}{2}}(\partial \Omega)}).
\label{cot1}
\end{eqnarray}
We now notice that
\begin{eqnarray}
|\hat{w}_{n+1}|^2-|H_{n+1}|^2=(\hat{w}_{n+1}-H_{n+1}) \bar{\hat{w}}_{n+1}+H_{n+1}(\bar{\hat{w}}_{n+1}-\bar{H}_{n+1}),
\label{for1}
\end{eqnarray}
where both $H_{n+1}$ and $\bar{\hat{w}}_{n+1}$ belong to $L^\infty(\Omega)$ because of (\ref{reg2}),(\ref{acotK2}) and Hille-Yoshida theorem using that, from (\ref{wn1}),
\begin{eqnarray}
\hat{w}_{n+1}&=&-H_n+K [\partial H_n+\partial H_{n+1}]+(I-\frac{i \tau}{2} A_0)^{-1}\bigg[2 H_n- K[\partial H_n+\partial H_{n+1}]\bigg].
\nonumber
\end{eqnarray}
Therefore,
$$
\| |\hat{w}_{n+1}|^2-|H_{n+1}|^2\|_{L^2(\Omega)}\le C \|\hat{w}_{n+1}-H_{n+1}\|_{L^2(\Omega)},$$
which implies that
$$
\| f(|\hat{w}_{n+1}|^2)-f(|H_{n+1}|^2)\|_{L^2(\Omega)}\le C \|\hat{w}_{n+1}-H_{n+1}\|_{L^2(\Omega)},$$
and thus, using that for any $\theta_1$ and $\theta_2$,
\begin{eqnarray}
|e^{i \theta_1}-e^{i \theta_2}|\le |\theta_1-\theta_2|,
\label{theta12}
\end{eqnarray}
it follows that
\begin{eqnarray}
\| e^{i \frac{\tau}{2}f(|\hat{w}_{n+1}|^2)}-e^{i \frac{\tau}{2}f(|H_{n+1}|^2)}\|_{L^2(\Omega)}\le C \tau \|\hat{w}_{n+1}-H_{n+1}\|_{L^2(\Omega)}.
\label{cot2}
\end{eqnarray}
Then, using (\ref{HH}),
\begin{eqnarray}
\hat{u}_{n+1}&=&e^{i \frac{\tau}{2} f(|\hat{w}_{n+1}|^2)} \hat{w}_{n+1}
\nonumber \\
&=& e^{i \frac{\tau}{2} f(|\hat{w}_{n+1}|^2)}(\hat{w}_{n+1}-H_{n+1})+\big[e^{i \frac{\tau}{2}f(|\hat{w}_{n+1}|^2)}-e^{i \frac{\tau}{2}f(|H_{n+1}|^2)} \big]H_{n+1}+u(t_{n+1}),  \nonumber
\end{eqnarray}
from what (\ref{cotaloc}) follows using (\ref{cot1}) and the fact that $\partial A^j u(t_n)$ and $g^{(j)}(t_n)$ are closely related through (\ref{nlse}), and then considering (\ref{cot2}) and that $e^{i \frac{\tau}{2}f(|\hat{w}_{n+1}|^2)}$ and $H_{n+1}$ belong to $L^\infty(\Omega)$.
\end{proof}



\subsection{Global error }

Although quite standard, we include here a thorough analysis for the global error for the sake of completeness. We will consider the notation $e_n=u(t_n)-u_n$.

\begin{theorem} Under the assumptions of Theorem \ref{thloc} and assuming that the numerical solution also remains bounded in $L^\infty(\Omega)$, it happens that
\begin{eqnarray}
\|e_n\|_{L^2(\Omega)} \le C e^{C' t_n} \tau^2 \max_{0 \le t \le T} (\|u(t)\|_{H^6(\Omega)}+\|\dot{u}\|_{H^4(\Omega)}+\max_{0\le j \le 3} \|u^{(j)}(t)\|_{L^2(\Omega)} \nonumber \\
\hspace{4cm}+\|u(t)\|_{H^4(\partial \Omega)}),
\label{cotaglob2}
\end{eqnarray}
where $C'$ depends on the bound of the numerical solution.
\label{thglob2}
\end{theorem}

\begin{proof}
As already argued in the proof of Theorem \ref{thloc} (see (\ref{for1})), whenever $v_1,v_2 \in L^\infty(\Omega)$,
\begin{eqnarray}
\||v_1|^2-|v_2|^2\|_{L^2(\Omega)}\le C' \|v_1-v_2\|_{L^2(\Omega)},
\nonumber
\end{eqnarray}
where $C'$ depends on the boundedness of $v_1$ and $v_2$. Then, using that $f\in C^1(\mathbb{R},\mathbb{R})$, for another constant $C'$ which also depends on the boundedness of $f'$ in the interval where $|v_1|^2$ and $|v_2|^2$ stay, it follows that
\begin{eqnarray}
\|f(|v_1|^2)-f(|v_2|^2)\|_{L^2(\Omega)}\le C' \|v_1-v_2\|_{L^2(\Omega)}.
\label{cotfl2}
\end{eqnarray}
Then, using that
\begin{eqnarray}
e^{i \frac{\tau}{2}f(|v_1|^2)}v_1-e^{i \frac{\tau}{2}f(|v_2|^2)}v_2= e^{i \frac{\tau}{2}f(|v_1|^2)}(v_1-v_2)+\big[ e^{i \frac{\tau}{2}f(|v_1|^2)}-e^{i \frac{\tau}{2}f(|v_2|^2)}\big]v_2,
\nonumber
\end{eqnarray}
and also (\ref{theta12}) and (\ref{cotfl2}), it follows that
$$e^{i \frac{\tau}{2}f(|v_1|^2)}v_1-e^{i \frac{\tau}{2}f(|v_2|^2)}v_2= e^{i \frac{\tau}{2}f(|v_1|^2)}(v_1-v_2)+\frac{\tau}{2}E(v_1,v_2),$$
where
\begin{eqnarray}
\|E(v_1,v_2)\|_{L^2(\Omega)} \le C' \|v_1-v_2\|_{L^2(\Omega)},
\label{acotE}
\end{eqnarray}
for some other constant $C'$. From here,
\begin{eqnarray}
e_{n+1}&=&(u(t_{n+1})-\hat{u}_{n+1})+(\hat{u}_{n+1}-u_{n+1})
\nonumber \\
&=&\rho_{n+1}+[e^{i \frac{\tau}{2}f(|\hat{w}_{n+1}|^2)}\hat{w}_{n+1}-e^{i \frac{\tau}{2}f(|w_{n+1}|^2)}w_{n+1}] \nonumber \\
&=& \rho_{n+1}+e^{i \frac{\tau}{2}f(|\hat{w}_{n+1}|^2)}(\hat{w}_{n+1}-w_{n+1})+\frac{\tau}{2} E(\hat{w}_{n+1},w_{n+1}). \label{recur}
\end{eqnarray}
Now, using (\ref{wn1}), for the stability function of midpoint rule (\ref{r}),
\begin{eqnarray}
\hat{w}_{n+1}-w_{n+1}&=&r(i \tau A_0)[e^{i \frac{\tau}{2}f(|u(t_n)|^2)}u(t_n)-e^{i \frac{\tau}{2}f(|u_n|^2)}u_n] \nonumber \\
&=&r(i \tau A_0)[e^{i \frac{\tau}{2}f(|u(t_n)|^2)} e_n+\frac{\tau}{2} E(u(t_n), u_n)], \nonumber
\end{eqnarray}
which, inserted in (\ref{recur}), implies that
$$
e_{n+1}=\rho_{n+1}+e^{i \frac{\tau}{2}f(|\hat{w}_{n+1}|^2)} r(i \tau A_0) e^{i \frac{\tau}{2}f(|u(t_n)|^2)}e_n+\tau \bar{E}(u(t_n), u_n),$$
where, using (\ref{acotE}) and that $\|r(i \tau A_0)\|_{L^2(\Omega)}=1$,
\begin{eqnarray}
\|\bar{E}(u(t_n), u_n)\|_{L^2(\Omega)}\le C' \|e_n\|_{L^2(\Omega)}.
\label{cotF}
\end{eqnarray}
Using that $e_0=0$, it follows inductively that
\begin{eqnarray}
e_n&=&\sum_{l=1}^n \bigg[\prod_{m=l}^{n-1} e^{i \frac{\tau}{2}f(|\hat{w}_{m+1}|^2)}r(i \tau A_0) e^{i \frac{\tau}{2}f(|u(t_m)|^2)} \bigg]\rho_l
\nonumber \\
&&+\tau \sum_{l=0}^{n-1} \bigg[\prod_{m=l+1}^{n-1} e^{i \frac{\tau}{2}f(|\hat{w}_{m+1}|^2)}r(i \tau A_0) e^{i \frac{\tau}{2}f(|u(t_m)|^2)}\bigg] \bar{E}(u(t_l),u_l).
\nonumber
\end{eqnarray}
Taking norms, using that $\|r(i \tau A_0)\|_{L^2(\Omega)}=1$ and (\ref{cotF}), it follows that
$$
\|e_n\|_{L^2(\Omega)} \le n \max_{l=1,\dots,n} \|\rho_l\|_{L^2(\Omega)}+\tau C' \sum_{l=0}^{n-1} \|e_l\|_{L^2(\Omega)},
$$
and from discrete Gronwall inequality the result follows.
\end{proof}



\begin{remark}
As stated in the abstract,
local order $3$ and thus global order $2$ can also be obtained without resorting to numerical differentiation in other more general problems with the same (or very similar) modified Strang method described above. This is thoroughly described in the appendix considering an abstract framework of Banach spaces which includes regular enough reaction-diffusion problems under time-dependent Dirichlet boundary conditions (c.f. \cite{ACR}).
\end{remark}

\subsection{Implementation}
\label{implem}

When implementing this by using a general space discretization in which an $N\times N$ matrix $A_{N,0}$ approximates $A_0$ when applied to some coefficients associated to a numerical solution with vanishing boundary conditions, and where $K g$ is approximated by the solution of the system
\begin{eqnarray}
A_{N,0} U_h+ C_N g=0, \nonumber
\end{eqnarray}
i.e by $-A_{N,0}^{-1} C_N g$, for some operator $C_N: H^\frac{3}{2}(\partial \Omega) \to \mathbb{R}^N$, the stage (\ref{wn1}) reads as follows
\begin{eqnarray}
W_N^{n+1}&=&-W_N^n-A_{N,0}^{-1} C_N [g^n(0)+g^n(\tau)] \nonumber \\
&&+(I-\frac{i \tau}{2} A_{N,0})^{-1}\bigg[2 W_N^n+A_{N,0}^{-1} C_N[g^n(0)+g^n(\tau)]\bigg]
\nonumber \\
&=&-W_N^n +(I-\frac{i \tau}{2} A_{N,0})^{-1}\bigg[2 W_N^n+A_{N,0}^{-1} C_N[g^n(0)+g^n(\tau)] \nonumber \\
&& \hspace{5cm}
-(I-\frac{i \tau}{2} A_{N,0})A_{N,0}^{-1} C_N[g^n(0)+g^n(\tau)] ]\bigg] \nonumber \\
&=&-W_N^n +(I-\frac{i \tau}{2} A_{N,0})^{-1}\bigg[2 W_N^n+\frac{i \tau}{2} C_N[g^n(0)+g^n(\tau)]\bigg]. \label{formula}
\end{eqnarray}
We notice that this implies solving just one linear system at this stage.

Then, each step of the modified Strang method consists of applying (\ref{formula}) with
$$W_N^n=e^{i \frac{\tau}{2} f(|U_N^{n}|^2)} U_N^n,$$
 where $U_N^n$ contains the information of the numerical approximation to $u(t_n)$ at the interior nodes, and then to calculate
$$U_N^{n+1}=e^{i \frac{\tau}{2} f(|W_N^{n+1}|^2)}W_N^{n+1},$$
as the approximation to $u(t_{n+1})$.

We also remark that the symmetry of the interior stage (\ref{wn1}) is conserved after space discretization and that the exact resolution of the symmetric initial and final stages lead to the symmetry of the whole step of the modified Strang method which is suggested here.

\section{Modified splitting methods of arbitrary order}
\label{Ssplitanyorder}

In this section we will justify that, by using as basis method the modified Strang method of the previous one and applying Yoshida's idea \cite{Y} of concatenating suitable timestepsizes of that basis method, arbitrary order can be achieved when integrating regular enough solutions of NLS equation under time-dependent Dirichlet boundary conditions.

More precisely, in a similar way to \cite{Y}, we will denote by $\tilde{S}_{2nd}(\tau)u_n$ the modified Strang method which consists of applying (\ref{midstrang}), (\ref{wn1}), (\ref{finalu}) when advancing a stepsize $\tau$ in the numerical integration of (\ref{nlse}).

As stated in \cite{Y}, by using Baker-Campbell-Hausdorff formula (BCH), when integrating with Strang method
$$
\dot{u}=(C+D)u$$
for two non-commutative operators $C$ and $D$,  advancing a stepsize $\tau$ consists of applying $S_{2nd}(\tau)u_n$ where
\begin{eqnarray}
S_{2nd}(\tau)=e^{\tau(C+D)+\tau^3 \alpha_3+\tau^5 \alpha_5+\tau^7 \alpha_7+\dots},
\label{S2}
\end{eqnarray}
where
\begin{eqnarray}
\alpha_3:= \frac{1}{12}[D,D,C]-\frac{1}{24}[C,C,D], \quad \alpha_5:= \frac{7}{5760}[C,C,C,C,D], \dots
\label{alphas}
\end{eqnarray}
(Here the notation of the commutator $[X,Y]:=XY-YX$ is used as well as $[X,X,Y]=[X,[X,Y]]$ for higher order commutators.) Then, there exist coefficients $w_0,w_1,\dots,w_m$, for which
\begin{eqnarray}
S^{(m)}(\tau):= S_{2nd}(w_m \tau)\dots S_{2nd}(w_1 \tau)S_{2nd}(w_0 \tau)S_{2nd}(w_1 \tau)\dots S_{2nd}(w_m \tau)
\label{Sm}
\end{eqnarray}
leads to higher order methods. Those coefficients are deduced by applying algebraically  BCH formula starting from the corresponding expressions $S_{2nd}(w_j \tau)$ ($j=0,\dots,m$) in (\ref{S2}).

Our suggestion in this paper is thus to consider
\begin{eqnarray}
\tilde{S}^{(m)}(\tau):= \tilde{S}_{2nd}(w_m \tau)\dots \tilde{S}_{2nd}(w_1 \tau)\tilde{S}_{2nd}(w_0 \tau)\tilde{S}_{2nd}(w_1 \tau)\dots \tilde{S}_{2nd}(w_m \tau)
\label{Stildem}
\end{eqnarray}
when advancing a stepsize $\tau$ in the numerical integration of (\ref{nlse}).

We firstly notice that, when the solution $u$ in (\ref{nlse}) satisfies
$$
u\in C^{(p+1)}([0,T],L^2(\Omega)),
$$
with
$$
g\in C^{(p+1)}([0,T],H^{\frac{3}{2}}(\partial \Omega)),
$$
the expression which is obtained through (\ref{midstrang}), (\ref{wn1}), (\ref{finalu}) when starting from $u_n=u(t_n)$ has an asymptotic expansion on powers of $\tau$ with a remainder which is of size $O(\tau^{p+1})$ in the $L^2$-norm.
The same happens for the exact solution, so the local error also has an asymptotic expansion in terms of powers of $\tau$.

We will now take into account that (\ref{nlse}) could have been discretized in space by
\begin{eqnarray}
\dot{U}_N=i (A_{N,0}U_N+ C_N g)+i f(|U_N|.^2). U_N,
\label{nlsed}
\end{eqnarray}
where the operators $A_{N,0}$ and $C_N$ are those described in Subsection \ref{implem} and $\cdot$ denotes the pointwise product.
If we now neglect the error coming from the space discretization, we can apply the theory for ODEs in \cite{Hairer_IG}, which assures that
the numerical solution formally satisfies a modified differential equation where the vector field depends on the timestepsize $\tau$. Moreover, as the method is symmetric and of second order, that vector field is a perturbation of the differential system in (\ref{nlsed}) where just even powers of $\tau$ turn up. Because of that, the modified Strang method in Section 4 after space discretization can be represented by the operator
$$
\tilde{S}_{2nd,N}(\tau)=e^{\tau(\tilde{C}_N+\tilde{D}_N)+\tau^3 \tilde{\alpha}_{3,N}+\tau^5 \tilde{\alpha}_{5,N}+\tau^7 \tilde{\alpha}_{7,N}+\dots},$$
where
$$\tilde{C}_N=A_{N,0}+C_N \partial_N, \quad \tilde{D}_N U_N=i f(|U_N|.^2). U_N,$$
and $\tilde{\alpha}_{3,N}$, $\tilde{\alpha}_{5,N}$, $\tilde{\alpha}_{7,N}$ are some operators which are in some way related to the coefficients of the local error of the modified Strang method, but which do not coincide with those in (\ref{alphas}) for $\tilde{C}_N$ and $\tilde{D}_N$.

We notice that, although $\tilde{D}_N$ is not linear, the Magnus expansion has also sense for nonlinear operators \cite{BCOR} and its convergence is valid for small enough $\tau$ \cite{AG}. Then, the algebraic expressions which lead to deduce the values $w_0,\dots,w_m$ in order to obtain order $p$ in Yoshida's paper \cite{Y} also lead now to the same order when considering (\ref{Stildem}).

The fact that global order $p$ is obtained when local error is $O(\tau^{p+1})$ is not detailed here for the sake of simplicity, but a similar (but more intricate proof) to that of Theorem \ref{thglob2}
can be done. In any case, we want to remark here that it is important that the eigenvalues of $i A_{N,0}$ are imaginary. As some of the values $w_j$ $(j=0,\dots,m)$ are always negative when searching for order $\ge 4$, considering other equations (like parabolic equations) or space discretizations for which the eigenvalues of $i A_{N,0}$ had negative real part would necessarily lead to instabilities  because $\|r(w \tau i A_{N,0})^n\|$ would not be bounded as $n$ increases any more.

\section{Numerical experiments}

\begin{table}[t]
\def\arraystretch{1.1}
\caption{Local errors (discrete $L^{2}$ norm) and estimated orders of convergence for three different splitting methods: Strang + spectral, Yoshida-Strang 4 + spectral and Yoshida-Strang 6 + spectral. Each method uses its own sequence of time steps optimized for local error analysis.\label{tab:errores_locales_datos}}
\begin{center}
\vspace{-3mm}
\small
\begin{tabular}{rrrrrrrrrr}
\hline
\multicolumn{3}{c}{Strang} & \multicolumn{3}{c}{Yoshida-Strang 4} & \multicolumn{3}{c}{Yoshida-Strang 6} \\
\multicolumn{1}{c}{step size} & \multicolumn{1}{c}{$L^{2}$ error} & \multicolumn{1}{c}{ord} &
\multicolumn{1}{c}{step size} & \multicolumn{1}{c}{$L^{2}$ error} & \multicolumn{1}{c}{ord} &
\multicolumn{1}{c}{step size} & \multicolumn{1}{c}{$L^{2}$ error} & \multicolumn{1}{c}{ord} \\
\hline
1.00e-03 & 6.412e-09 & --       & 1.000e-05 & 8.860e-20 & --        & 1.000e-05 & 2.441e-21 & --         \\
5.00e-04 & 8.011e-10 & 3.0      & 5.000e-06 & 6.079e-21 & 3.9      & 5.000e-06 & 6.841e-23 & 5.2       \\
2.50e-04 & 1.001e-10 & 3.0      & 2.500e-06 & 2.549e-22 & 4.6      & 2.500e-06 & 7.686e-25 & 6.5       \\
1.25e-04 & 1.251e-11 & 3.0      & 1.250e-06 & 8.667e-24 & 4.9      & 1.250e-06 & 6.608e-27 & 6.9       \\
6.25e-05 & 1.564e-12 & 3.0      & 6.250e-07 & 2.768e-25 & 5.0      & 6.250e-07 & 5.283e-29 & 7.0       \\
\hline
\end{tabular}
\vspace{-3mm}
\end{center}
\end{table}

\begin{table}[t]
\def\arraystretch{1.1}
\caption{Local errors ($L^{\infty}$ norm) and estimated orders of convergence for three different splitting methods: Strang + spectral, Yoshida-Strang 4 + spectral and Yoshida-Strang 6 + spectral. Each method uses its own sequence of time steps optimized for local error analysis.\label{tab:errores_locales_datosinf}}
\begin{center}
\vspace{-3mm}
\small
\begin{tabular}{rrrrrrrrrr}
\hline
\multicolumn{3}{c}{Strang} & \multicolumn{3}{c}{Yoshida-Strang 4} & \multicolumn{3}{c}{Yoshida-Strang 6} \\
\multicolumn{1}{c}{step size} & \multicolumn{1}{c}{$L^{\infty}$ error} & \multicolumn{1}{c}{ord} &
\multicolumn{1}{c}{step size} & \multicolumn{1}{c}{$L^{\infty}$ error} & \multicolumn{1}{c}{ord} &
\multicolumn{1}{c}{step size} & \multicolumn{1}{c}{$L^{\infty}$ error} & \multicolumn{1}{c}{ord} \\
\hline
1.00e-03 & 5.953e-09 & --       & 1.000e-05 & 8.635e-19 & --        & 1.000e-05 & 2.283e-20 & --         \\
5.00e-04 & 7.443e-10 & 3.0      & 5.000e-06 & 5.838e-20 & 3.9      & 5.000e-06 & 6.371e-22 & 5.2       \\
2.50e-04 & 9.305e-11 & 3.0      & 2.500e-06 & 2.434e-21 & 4.6     & 2.500e-06 & 7.153e-24 & 6.5       \\
1.25e-04 & 1.163e-11 & 3.0      & 1.250e-06 & 8.261e-23 & 4.9      & 1.250e-06 & 6.149e-26 & 6.9       \\
6.25e-05 & 1.454e-12 & 3.0      & 6.250e-07 & 2.638e-24 & 5.0      & 6.250e-07 & 4.917e-28 & 7.0       \\
\hline
\end{tabular}
\vspace{-3mm}
\end{center}
\end{table}

In this section we will corroborate the previous results by integrating a regular enough solution of NLS equation in a bounded space interval with time-dependent boundary conditions. More particularly, we will consider (\ref{nlse}) with $f(x)=8x$ , $\Omega=(-1,1)$ and $u_0$ and $g$ such that the exact solution is
$$
u(x,t)=e^{i t} \mbox{sech}(x) \frac{1+\frac{3}{4}\mbox{sech}(x)^2(e^{8it}-1)}{1-\frac{3}{4}\mbox{sech}(x)^4 \sin(4t)^2}.$$

In such a way, the hypotheses of Theorems \ref{thloc} are satisfied. We have integrated this problem till time $T=1$ considering a collocation Gauss-Lobatto spectral method \cite{BM,CHQZ} which is based on $N=50$ nodes, so that the error in space can be considered negligible even with quadruple precision machine accuracy.

In a first place, we have integrated the problem in time with the modified Strang method using the implementation which is well described in Subsection \ref{implem}. Then, we have measured the error in the $L^2$-norm using the Gauss-Lobatto quadrature rule which is based on the same nodes of the space discretization. In the first three columns of Table \ref{tab:errores_locales_datos}, we can see that local order $3$ turns up, and we have managed that local order without any differentiation of data. Although  not theoretically proved, the same happens when using the maximum norm of the error on the same Gauss-Lobatto nodes, as it is shown in the first three columns of Table \ref{tab:errores_locales_datosinf}. Then, the second order on the global error, which is stated in Theorem \ref{thglob2} for the $L^2$-norm, can be observed in Figures \ref{f1}, and also in Figure \ref{f2} for the $L^\infty$-norm.

In a second place, we have integrated (\ref{nlse}) in time with two modified splitting methods which are based on two symmetric splitting methods deduced by Yoshida from Strang method. More particularly, the $4$th-order method in Section 4 of \cite{Y} for which $m=1$ in  (\ref{Stildem}) and
$$
w_0=-\frac{2^{\frac{1}{3}}}{2-2^{\frac{1}{3}}}, \quad w_1=\frac{1}{2-2^{\frac{1}{3}}},$$
and the $6$th-order method in the last column of Table 1 in the same paper for which $m=3$.

We notice that the solution is regular enough so that all hypotheses in Section \ref{Ssplitanyorder} of this paper apply. Then, we can corroborate in the last columns of Tables \ref{tab:errores_locales_datos} and \ref{tab:errores_locales_datosinf} that local orders $5$ and $7$ are respectively obtained when sufficiently decreasing the time stepsize. This explains the global respective orders $4$ and $6$ which both methods tend to show in Figures \ref{f1} and \ref{f2} when considering not so small time stepsizes. We want to remark here that each step of the 4th-order method is three times more expensive than each step of Strang method because $m=1$ and each step of the 6th-order method is seven times more expensive than each step of Strang method because $m=3$. However, even taking that into account, in order to obtain an error like $10^{-4}$, the 4th and 6th order methods would be cheaper than Strang method. Moreover, the higher the accuracy which is searched for, the more the advantages of higher order methods against the lower ones.

As a conclusion, we remark again the big gain in efficiency which we have obtained with the modification of splitting methods which are based on Strang integrator. Never in the literature local order $\ge 3$ had been obtained when integrating (\ref{nlse}) with general time-dependent boundary conditions using the favourable properties of splitting methods and without using neither analytical nor numerical differentiation of data.  In this paper, we manage to get local and global order as high as desired under those conditions, and what is more, in a computionally efficient way.

\begin{figure*}
\centerline{\includegraphics[width=100mm]{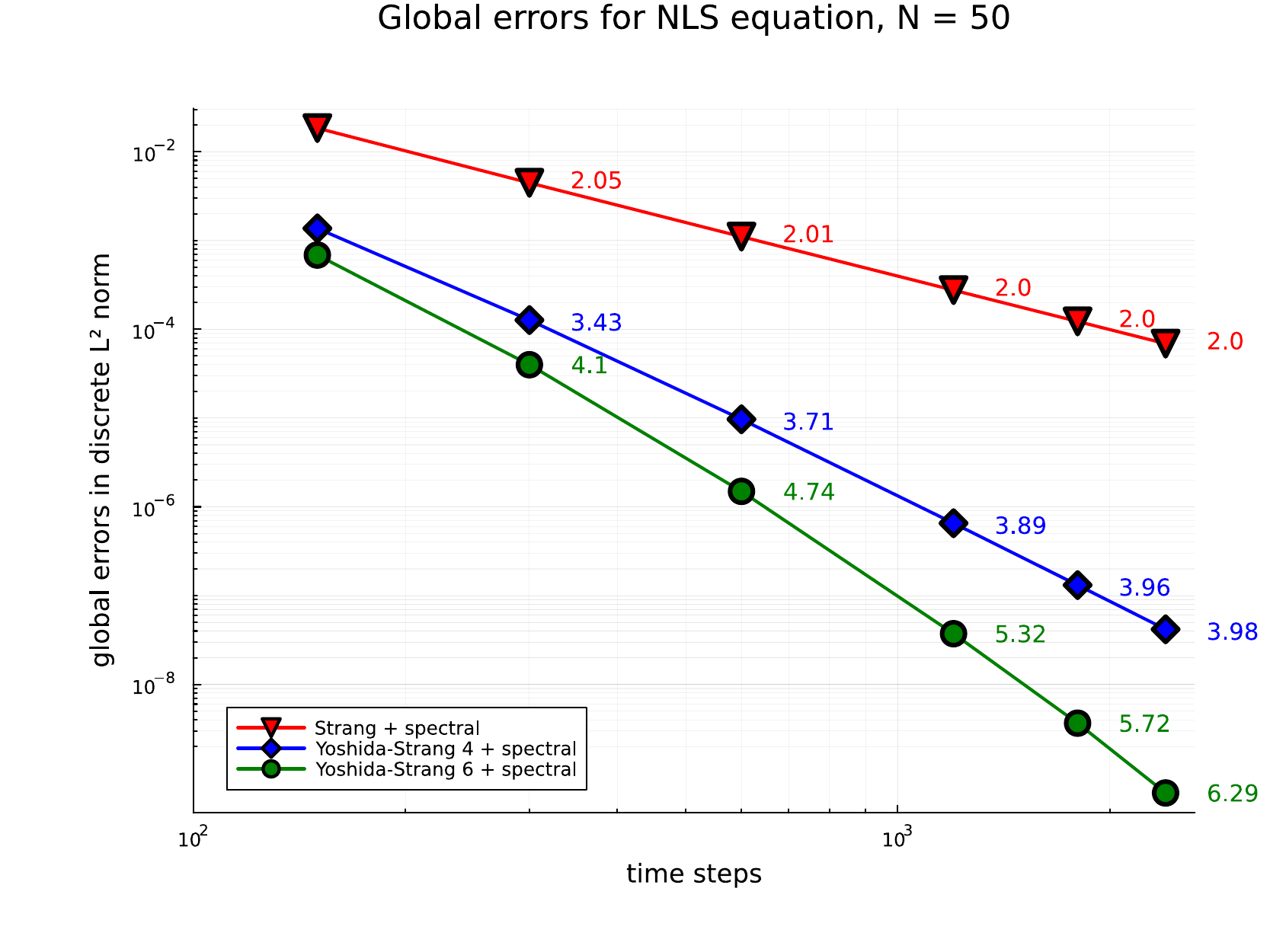}}
\caption{$L^2$-error against number of time steps when integrating problem (\ref{nlse}) with modified Strang method and modified Yoshida splitting methods of orders 4 and 6.} \label{f1}
\end{figure*}

\begin{figure*}
\centerline{\includegraphics[width=100mm]{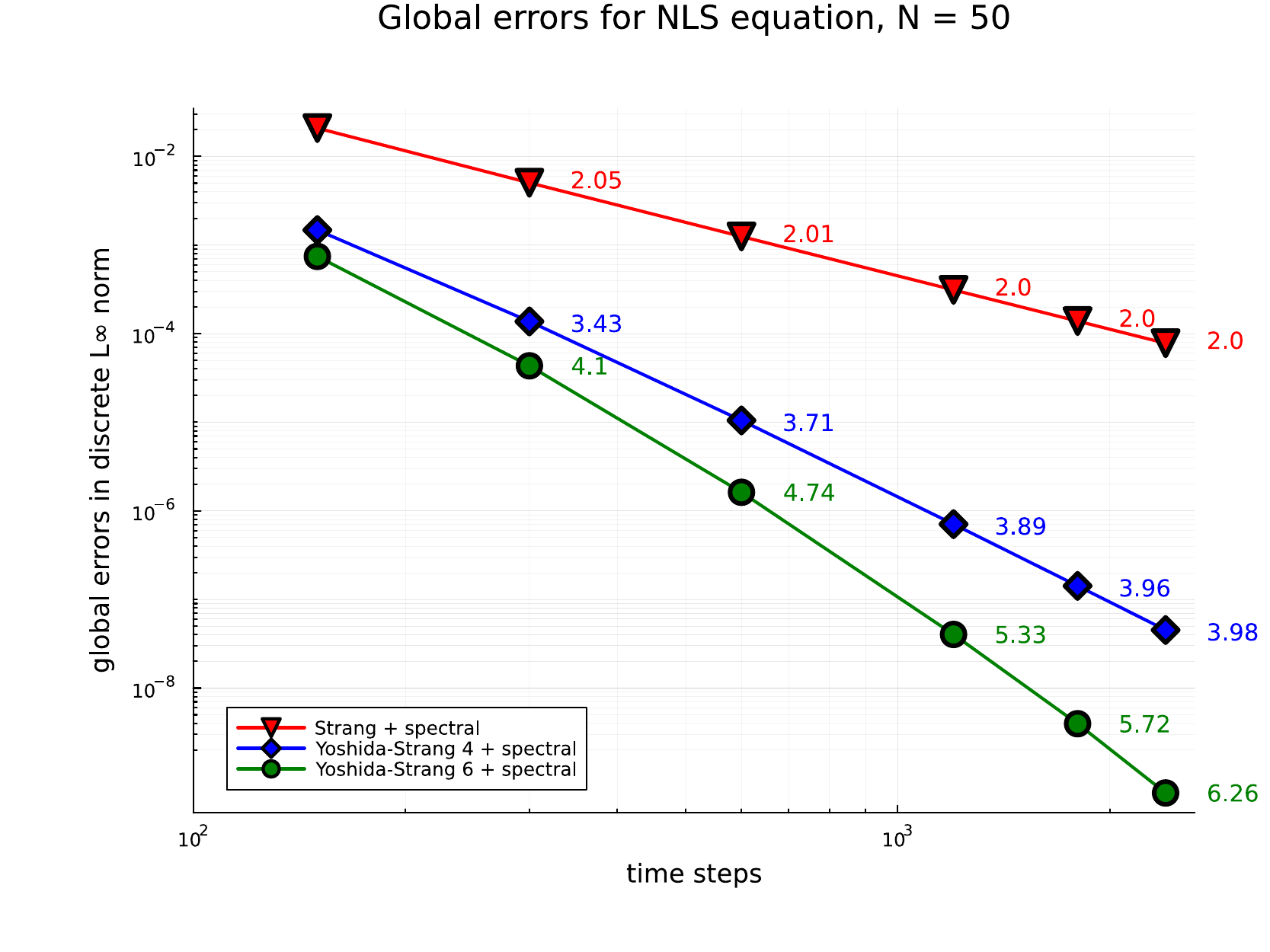}}
\caption{Maximum-error against number of time steps when integrating problem (\ref{nlse}) with modified Strang method and modified Yoshida splitting methods of orders 4 and 6. } \label{f2}
\end{figure*}

\section{Appendix}

In this appendix we consider the abstract non-homogeneous initial boundary value problem
\begin{eqnarray}
\dot{u}(t)&=& A u(t) +h(t,u(t)), \quad 0 \le t \le T,\nonumber \\
u(0)&=&u_0\in X, \nonumber \\
\partial u(t)&=&g(t)\in Y, \label{abstract}
\end{eqnarray}
where $T>0$, $A: D(A) \subset X \to X$ is a linear operator, $X$ and $Y$ are Banach spaces with respective norms $\|\cdot\|_X$ and $\|\cdot\|_Y$, $h$ is such that
the initial value problem
\begin{eqnarray}
\dot{v}(s)&=&h(t+s, v(s)), \nonumber \\
v(0)&=&v_0, \label{ec_h}
\end{eqnarray}
is well-posed not only on $X$ but also on $Y$, with $\partial: X \to Y$ a linear operator for which it happens that, when $v_0 \in X$, $\partial v(s)$ is the solution of (\ref{ec_h}) with initial condition $\partial v_0$. This is typical of Dirichlet boundary conditions when $h$ is a smooth function of its arguments which does not contain derivatives.

According to Remark 3 in \cite{ACR}, the well-posedness of (\ref{abstract}) is guaranteed under the following hypotheses:

\begin{enumerate}
\item[(A1)] The boundary operator $\partial: D(A)\subset X \to Y$ is onto.

\item[(A2)] $\mbox{Ker}(\partial)$ is dense in $X$ and $A_0: D(A_0)= \mbox{ker}(\partial)\subset X \to X$, the restriction of $A$ to $\mbox{ker}(\partial)$, is the infinitesimal generator of a $C_0$-semigroup $\{ e^{t A_0} \}_{t \ge 0}$ in $X$, which type $\omega$ is assumed to be negative. (In case it were not negative, the problem can be reduced to one of negative type using the procedure well described in Section 2 of \cite{ACP}.)
\item[(A3)] If $z\in \mathbb{C}$ satisfies $\mbox{Re}(z)>\omega$ and $v\in Y$, then the steady state problem
\begin{eqnarray}
A x&=& z x, \nonumber \\
\partial x&=& v, \nonumber
\end{eqnarray}
possesses a unique solution denoted by $x=K(z)v$. Moreover, the linear operator $K(z):Y \to D(A)$ satisfies
$$
\|K(z) v\|_X \le C \|v\|_Y,$$
where the constant holds for any $z$ such that $\mbox{Re}(z) \ge \omega_0 > \omega$.
\item[(A4)] The nonlinear source $h$ belongs to $C^1([0,T]\times X, X)$.
\end{enumerate}

The main difference with the strategy on the NLS equation is that (\ref{ec_h}) is not exactly solvable any more. Because of that, we must suggest a numerical method to integrate it. We will assume no order reduction turns up in that part of the integration, as is the case when integrating non-stiff ODE systems. More precisely, we will consider a method of local order $3$ to integrate that part and will denote it by $\Psi_{\tau}^{h,t}(v_0)$ when advancing a stepsize $\tau$ from (\ref{ec_h}). We will also assume that this integrator imitates the invariance with respect to $\partial$, i.e.,
\begin{eqnarray}
\partial \Psi_{\tau}^{h,t}(v_0)=\Psi_{\tau}^{h,t}(\partial v_0). \label{frontera_inv}
\end{eqnarray}
Then, each step of the modified Strang method starting from $u_n$ consists of
\begin{enumerate}
\item[1.] Integrating  (\ref{ec_h}) with $v_0=u_n$ and $t=t_n$ so that $\Psi_{\frac{\tau}{2}}^{h,t_n}(u_n)$ is computed.
\item[2.] Integrating till time $\tau$
\begin{eqnarray}
\dot{w}^n(s)&=& A w^n(s), \nonumber \\
w^n(0)&=&\Psi_{\frac{\tau}{2}}^{h,t}(u_n), \nonumber \\
\partial w^n (s)&=& g^n(s), \nonumber
\end{eqnarray}
where
$$
g^n(0)=\Psi_{\frac{\tau}{2}}^{h,t_n}(g(t_n)), \quad g^n(\tau)=\Psi_{-\frac{\tau}{2}}^{h,t_{n+1}}(g(t_{n+1})),
$$
using the rational-like method which is based on the midpoint rule, i.e, similarly to (\ref{wn1}), the following formula is implemented
\begin{eqnarray}
w_{n+1}&=&-\Psi_{\frac{\tau}{2}}^{h,t_n}(u_n)+K(0)[g^n(0)+g^n(\tau)] \nonumber \\
&&+(I-\frac{\tau}{2}A_0)^{-1}\bigg[2 \Psi_{\frac{\tau}{2}}^{h,t_n}(u_n)-K(0)[g^n(0)+g^n(\tau)]\bigg]. \nonumber
\end{eqnarray}
\item[3.] Integrating  (\ref{ec_h}) with $v_0=w_{n+1}$ and $t=t_n+\frac{\tau}{2}$ so that $\Psi_{\frac{\tau}{2}}^{h,t_n+\frac{\tau}{2}}(w_{n+1})$ is computed.
\end{enumerate}

Then, we have the following results corresponding to both the local and global error when using this procedure.

\begin{theorem} Let us assume that the solution of (\ref{abstract}) belongs to
$$C^3([0,T],X) \cap C([0,T], D(A^3)),$$
$h\in  C^2([0,T],X)$, $g\in  C^2([0,T],Y)$ and that $\Psi_{\tau}^{h,t}(u(t))\in D(A^3)$ for every $\tau, t>0$. Then, for small enough $\tau$,
$\|\rho_n\|_X=O(\tau^3)$ where the constant in Landau notation is independent of $\tau$ and $t_n$.

\label{thloc_ab}
\end{theorem}

\begin{proof}
The main difference with the proof of Theorem \ref{thloc} is that now $i A$ and $i A_0$ must be substituted by $A$ and $A_0$ and that
$$H_n= \Psi_{\frac{\tau}{2}}^{h,t_n}(u(t_n)), \quad H_{n+1}= \Psi_{-\frac{\tau}{2}}^{h,t_{n+1}}(u(t_{n+1})).$$
Using then the assumption of local order $3$ for $\Psi_{\tau}^{h,t}$, it follows by developing into Taylor series that
\begin{eqnarray}
H_n&=&u(t_n)+\frac{\tau}{2} h(t_n,u(t_n))+\frac{\tau^2}{8}[h_t(t_n,u(t_n))+h_u(t_n,u(t_n))h(t_n,u(t_n))]+O(\tau^3), \nonumber \\
H_{n+1}-H_n&=&\tau A u(t_n)+\frac{\tau^2}{2}[\ddot{u}(t_n)-h_t(t_n,u(t_n))-h_u(t_n,u(t_n))\dot{u}(t_n)]+O(\tau^3). \nonumber
\end{eqnarray}
From here, proceeding as in the proof of Theorem \ref{thloc}, it follows that
\begin{eqnarray}
\stackrel{\cdot}{\overbrace{(\hat{w}^n-w^{n,b})}}(s)&=& A (\hat{w}^n(s)-w^{n,b}(s))+s O(\tau)+\frac{s^2}{2} (A^3 u(t_n) +O(\tau)), \nonumber \\
(\hat{w}^n-w^{n,b})(0)&=&0, \nonumber \\
\partial (\hat{w}^n-w^{n,b})(s)&=&0, \nonumber
\end{eqnarray}
where the invariance property (\ref{frontera_inv}) has been used. As a consequence,
$$
\hat{w}^n(s)-w^{n,b}(s)=s^2 \varphi_2( s A_0) O(\tau)+ s^3 \varphi_3( s A_0) [\frac{1}{2} A^3 u(t_n)+O(\tau)],
$$
which implies, considering $s=\tau$, that $\|\hat{w}^n(\tau)-H_{n+1}\|_X=O(\tau^3)$. Also with the same argument as in the proof of Theorem \ref{thloc} but using the new hypothesis that $$H_n=\Psi_{\frac{\tau}{2}}^{h,t_n}(u(t_n))\in D(A^3),$$
 it also happens that $\|\hat{w}_{n+1}-\hat{w}^n(\tau)\|_X=O(\tau^3)$. Then, using the triangle inequality, $\|\hat{w}_{n+1}-H_{n+1}\|_X =O(\tau^3)$.

From this, using that
\begin{eqnarray}
\rho_n&=&\hat{u}_{n+1}-u(t_{n+1})= \Psi_{\frac{\tau}{2}}^{h,t_n+\frac{\tau}{2}}(\hat{w}_{n+1})-u(t_{n+1}) \nonumber \\
&=& \big[\Psi_{\frac{\tau}{2}}^{h,t_n+\frac{\tau}{2}}(\hat{w}_{n+1})-\Psi_{\frac{\tau}{2}}^{h,t_n+\frac{\tau}{2}}(H_{n+1})\big]+
\big[ \Psi_{\frac{\tau}{2}}^{h,t_n+\frac{\tau}{2}}(\Psi_{-\frac{\tau}{2}}^{h,t_{n+1}}(u(t_{n+1}))-u(t_{n+1}) \big], \nonumber
\end{eqnarray}
that $\Psi_{\tau}^{h,t}$ is of third order, that $h\in C^2([0,T]\times X,X)$ and the reversibility of the true flow of (\ref{ec_h}), the result follows.

\end{proof}

\begin{theorem} Under the assumptions of Theorem \ref{thloc_ab} and assuming also that there exists a constant $C$ such that
\begin{eqnarray}
\|r(\tau A_0)^n \|_{X} \le C \quad \mbox{ whenever }n \tau \le T,
\label{stab}
\end{eqnarray}
it happens that the global error $e_n=u(t_n)-u_n$ satisfies
$$\|e_n\|_X=O(\tau^2),$$
where the constant in Landau notation is independent of $\tau$ and is bounded for $0\le n \tau \le T$.
\end{theorem}

\begin{proof}
We consider the decomposition
\begin{eqnarray}
e_n&=&(u(t_{n+1})-\hat{u}_{n+1})+(\hat{u}_{n+1}-u_{n+1}) \nonumber \\
&=& \rho_{n+1}+\Psi_{\frac{\tau}{2}}^{h,t_n+\frac{\tau}{2}}(\hat{w}_{n+1})-\Psi_{\frac{\tau}{2}}^{h,t_n+\frac{\tau}{2}}(w_{n+1}) \nonumber \\
&=& \rho_{n+1}+\hat{w}_{n+1}-w_{n+1}+\frac{\tau}{2}E(\hat{w}_{n+1},w_{n+1})+O(\tau^3), \label{decomp}
\end{eqnarray}
where, for the last equality, we have used that $\Psi_{\tau}^{h,t}$ integrates with third local order and that, for the exact flow $\phi_{\tau}^{h,t}$, it happens that
$$
\phi_{\frac{\tau}{2}}^{h,t_n+\frac{\tau}{2}}(\hat{w}_{n+1})-\phi_{\frac{\tau}{2}}^{h,t_n+\frac{\tau}{2}}(w_{n+1})=\hat{w}_{n+1}-w_{n+1}+\frac{\tau}{2}E(\hat{w}_{n+1},w_{n+1}),$$
with
\begin{eqnarray}
\|E(\hat{w}_{n+1},w_{n+1})\|_X \le C \|\hat{w}_{n+1}-w_{n+1}\|_X,
\end{eqnarray}
for some constant $C$, using that $h\in C^1([0,T]\times X, X)$.

Then, by the definition of $\hat{w}_{n+1}$, $w_{n+1}$ and (\ref{wn1}),
\begin{eqnarray}
\hat{w}_{n+1}-w_{n+1}&=& r(\tau A_0)\big[ \Psi_{\frac{\tau}{2}}^{h,t_n+\frac{\tau}{2}}(\hat{w}_{n+1})-\Psi_{\frac{\tau}{2}}^{h,t_n+\frac{\tau}{2}}(w_{n+1}) \big] \nonumber \\
&=&r(\tau A_0)\big[e_n+\frac{\tau}{2}E(u(t_n),u_n)+O(\tau^2)\big], \label{whw}
\end{eqnarray}
from what there exists another constant $C$ such that
\begin{eqnarray}
\| \hat{w}_{n+1}-w_{n+1} \|_X \le C \|e_n\|_X.
\label{cotaw}
\end{eqnarray}
Inserting (\ref{whw}) in (\ref{decomp}), it follows that
\begin{eqnarray}
e_{n+1}=\rho_{n+1}+r(\tau A_0)[e_n+\frac{\tau}{2}E(u(t_n),u_n)+O(\tau^3)]+\frac{\tau}{2}E(\hat{w}_{n+1},w_{n+1})+O(\tau^3).
\nonumber
\end{eqnarray}
Therefore, it is inductively proved that
\begin{eqnarray}
e_n&=&r(\tau A_0)^n e_0+\sum_{l=1}^n r(\tau A_0)^{n-l} [\rho_l+\frac{\tau}{2}E(\hat{w}_l,w_l)+O(\tau^3)] \nonumber \\
&&+\frac{\tau}{2} \sum_{l=0}^{n-1} r(\tau A_0)^{n-l} [ E(u(t_l),u_l)+O(\tau^2)]. \nonumber
\end{eqnarray}
This implies, using that $e_0=0$, (\ref{stab}), (\ref{whw}) and (\ref{cotaw}) that, for $t_n \le T$,
$$
\|e_n\|_X \le C n \max_{l=0,\dots,n} \|\rho_l\|_X + \frac{\tau}{2} C \sum_{l=1}^n \|e_{l-1}\|_X+\frac{\tau}{2} C \sum_{l=0}^{n-1} \|e_l\|_X+O(\tau^2),$$
which implies the result using the bound for the local error in Theorem \ref{thloc_ab} and discrete Gronwall inequality.
\end{proof}

\begin{remark} Hypothesis (\ref{stab}) is well justified for sectorial operators in \cite{P}.
\end{remark}
\end{document}